\documentclass[sn-mathphys,Numbered]{sn-jnl}

\usepackage{graphicx}%
\usepackage{multirow}%
\usepackage{amsmath,amssymb,amsfonts}%
\usepackage{amsthm}%
\usepackage{mathrsfs}%
\usepackage[title]{appendix}%
\usepackage{xcolor}%
\usepackage{textcomp}%
\usepackage{manyfoot}%
\usepackage{booktabs}%
\usepackage{algorithm}%
\usepackage{algorithmicx}%
\usepackage{algpseudocode}%
\usepackage{listings}%

\theoremstyle{thmstyleone}%
\newtheorem{theorem}{Theorem}
\theoremstyle{thmstyletwo}%
\newtheorem{example}{Example}%
\newtheorem{remark}{Remark}%
\newtheorem{lemma}{Lemma}%
\theoremstyle{thmstylethree}%
\newtheorem{definition}{Definition}%
\newtheorem{corollary}{Corollary}%
\begin{document}

\title[Article Title]{Finite Time Stability Analysis for  Fractional Stochastic Neutral Delay Differential Equations}
\author[1]{\fnm{Javad} \sur{A. Asadzade}}\email{javad.asadzade@emu.edu.tr}

\author[2]{\fnm{Nazim} \sur{I. Mahmudov}}\email{nazim.mahmudov@emu.edu.tr}
\equalcont{These authors contributed equally to this work.}

\affil[1]{\orgdiv{Department of Mathematics}, \orgname{Eastern Mediterranean University}, \orgaddress{\street{} \city{Mersin 10, 99628, T.R.}, \postcode{5380},\country{North Cyprus, Turkey}}}

\affil[2]{\orgdiv{Department of Mathematics}, \orgname{Eastern Mediterranean University}, \orgaddress{\street{} \city{Mersin 10, 99628, T.R.}, \postcode{5380},\country{North Cyprus, Turkey}}}

\affil[2]{\orgdiv{	Research Center of Econophysics}, \orgname{Azerbaijan State University of Economics (UNEC)}, \orgaddress{\street{Istiqlaliyyat Str. 6}, \city{Baku }, \postcode{1001},  \country{Azerbaijan}}}

\abstract{In this manuscript, we investigate a fractional stochastic neutral differential equation with  time delay, which includes both deterministic and stochastic components.  Our primary objective is to rigorously prove the existence of a unique solution that satisfies given initial conditions. Furthermore, we extend our research to investigate the finite-time stability of the system by examining trajectory behavior over a given period. We employ advanced mathematical approaches to systematically prove finite-time stability, providing insights on convergence and stability within the stated interval. Using illustrative examples, we strengthen this all-encompassing examination into the complicated dynamics and stability features of fractionally ordered stochastic systems with time delays. The implications of our results extend to various fields, such as control theory, engineering, and financial mathematics, where understanding the stability of complex systems is crucial.}


\maketitle

\section{Introduction}\label{sec1}

Fractional differential equations extend the traditional ordinary and partial differential equations by allowing the order of differentiation to be any real or even complex number, rather than just a natural number. These equations are becoming more prevalent in modeling mathematical problems across various fields, including stability theory (\cite{4,7,9,10,11,14,15,17,21,22,25,27,28,29,31,41,42,44,48,52,53}), control theory (\cite{32,33,34,35,36,49,50}), stochastic analysis (\cite{2,3,4,8,9,10,11,12,13,14,15,16,17,18,20,27,31,43,44}), and positive time-continuous systems (\cite{1,37}). Recently, there has been a growing interest in fractional differential equations that incorporate the Caputo fractional time-derivative operator, driven by its applicability in numerous practical scenarios within science and engineering, as highlighted in recent literature (\cite{38}).

The incorporation of fractional derivatives provides a more accurate representation of memory and hereditary properties in various materials and processes, enhancing the modeling of phenomena such as viscoelasticity, anomalous diffusion, and signal processing. This has opened new avenues for research and application, leading to innovative solutions in both theoretical studies and practical implementations. As a result, fractional differential equations are now a critical tool in the modeling and analysis of complex systems across diverse scientific disciplines.

Given their ability to capture intricate system dynamics, fractional differential equations have also proven valuable in the study of finite-time stability. Finite-time stability analysis is particularly important when dealing with systems that need to achieve stability within a finite period, rather than asymptotically over an infinite time horizon. This type of stability is crucial in practical applications where immediate performance and response are required, such as in control systems, financial modeling, and various engineering processes.

Finite-time stability study is at the pinnacle of the field of stochastic differential equations (SDEs). It provides an understanding of the complex processes that occur within a short period. Finite-time stability research examines complex convergence and trajectory patterns in a short interval, while classic stability theories focus on asymptotic behavior over an infinite horizon.

The transformative contributions of K. Itô and H. P. McKean in the development of Itô calculus, a crucial tool for analyzing SDEs and comprehending the stability properties of stochastic systems, marked a significant milestone in this collective endeavor. Their work laid the foundation for a profound understanding of stochastic processes, which in turn facilitated significant advancements in the analysis of finite-time stability.

This confluence of mathematical pioneers collectively shaped the landscape of stability analysis, paving the way for further advancements in the theory of stochastic processes. Finite-time stability, gaining prominence, found applications in diverse fields including control theory, finance, and biology. In the realm of control theory and robotics, where temporal stability is paramount, controllers grounded in finite-time stability criteria offer robustness against uncertainties and disturbances. 

Ongoing research endeavors continue to propel the understanding of finite-time stability, with scholars developing novel mathematical tools and techniques to unravel intricacies and address challenges. The evolution of finite-time stability encapsulates a dynamic field, bridging deterministic stability concepts with the mathematical framework tailored for stochastic processes, thereby contributing to a profound understanding of systems characterized by both deterministic and stochastic components. (for extra information see, \cite{1}-\cite{54})

Neutral differential equations with a delay, also known as neutral differential-difference equations, are a type of ordinary or partial differential equation that includes both delayed values and delayed derivatives of the solution. These equations, both in their deterministic and stochastic forms, are extensively researched due to their significant applications in fields such as population dynamics, automatic control systems, and neural networks (\cite{51,52,53}).

Concerning the analytical solutions to neutral differential-difference equations, only a few papers have been published. Notably, Pospisil and Skripkova have derived an analytical representation of the initial value problem for linear systems of neutral differential equations with permutable matrices (\cite{54}).

This paper is devoted to the study of such linear systems of neutral differential equations with a delay. By assuming the linear parts to be given by pairwise permutable matrices, we derive a representation of the solution to a nonhomogeneous initial value problem using a matrix polynomial of a degree that depends on time. This approach provides a structured and efficient method for analyzing these complex systems, offering new insights and potential applications in various scientific and engineering fields.

Afterwards, in the fractional sense, Zhang et al. \cite{23} have investigated the representation of the general solution to a linear fractional neutral differential-difference system with a single constant delay. This further extends the understanding and applicability of neutral differential equations by incorporating fractional calculus, which provides a more comprehensive framework for modeling systems with memory and hereditary properties.

Subsequently,  Huseynov and Mahmudov (\cite{1}) addressed the problem depicted in Equation \eqref{IH}. Their study focused on the initial value problem associated with linear matrix coefficient systems of fractional-order neutral differential equations, encompassing two incommensurate constant delays in Caputo’s sense:

\begin{equation}\label{IH}
	\begin{cases}
		{^{C}}D^\lambda_0 y(t)= A_{0}y(t) + A_{1}y(t - h_{1})+ A_{2}D^\lambda_0 y(t - h_{2}) + f(t), \quad t\in[0,T],\\
		y(t) = \varphi(t), \quad t \in [-h, 0],\quad h=\max\{h_{1},h_{2}\}.
	\end{cases}
\end{equation}

Their research yielded exact analytical representations for solutions to both linear homogeneous and non-homogeneous neutral fractional-order differential-difference equation systems. This was accomplished by introducing newly defined delayed Mittag–Leffler type matrix functions.

Furthermore, their paper introduced a criterion for evaluating the positivity of a specific class of fractional-order linear homogeneous time-delay systems. They rigorously established the global existence and uniqueness of solutions for nonlinear fractional neutral delay differential equation systems using the contraction mapping principle within a weighted space of continuous functions, considering classical Mittag–Leffler functions.

Moreover, the authors obtained Ulam–Hyers stability results for solutions through a fixed-point approach. This comprehensive approach enhances our understanding of fractional-order neutral differential equations and provides valuable insights into their stability properties and solution behavior.

By motivated the above paper, we delve into the investigation of a neutral fractional order stochastic differential equation with a time delay, given by:

\begin{equation}\label{1}
	\begin{cases}
		{^{C}}D^\lambda_0 y(t)= A_{0}y(t) + A_{1}y(t - h_{1})+ A_{2}D^\lambda_0 y(t - h_{2}) + f(t, y(t), y(t - h_{1}),y(t - h_{2}))\\
		+\sigma(t, y(t), y(t - h_{1}),y(t - h_{2}))\frac{dW(t)}{dt}, \quad t\in[0,T],\\
		y(t) = \varphi(t), \quad t \in [-h, 0],\quad h=\max\{h_{1},h_{2}\}.
	\end{cases}
\end{equation}
\bigskip
Here are the key components of the equation:

$\bullet$ \( {^{C}}D^\lambda_0 y(t) \) represents a Caputo fractional derivative of order \( \lambda \) of the function \( y(t) \).

$\bullet$ \( A_{0}y(t) \), \( A_{1}y(t - h_{1}) \), and \( A_{2}D^\lambda_0 y(t - h_{2}) \) are linear terms in the system, where $A_0, A_1, A_2 \in \mathbb{R}^{n\times n}$.

$\bullet$ \( f(t, y(t), y(t - h_{1}),y(t - h_{2})) \) represents a nonlinear function of the system.

$\bullet$ \( \sigma(t, y(t), y(t - h_{1}),y(t - h_{2}))\frac{dW(t)}{dt} \) is a stochastic term involving a Wiener process \( W(t) \).

$\bullet$ \( y(t) = \varphi(t) \) sets the initial condition for the system in the interval \( [-h, 0] \).
\bigskip

In this paper, our main goal is to prove the existence and uniqueness of solutions for this system. By using a thorough mathematical analysis, we aim to demonstrate the model's well-posedness and the existence of a unique solution that complies with the initial conditions and system dynamics that are stated.

Moreover, we investigate finite-time stability for the fractional-order stochastic differential equation that is suggested. We examine the behavior of trajectories over a finite interval of time, providing significant details on elements of convergence and stability. We systematically demonstrate the system's finite-time stability by applying sophisticated mathematical approaches, which offer a clear comprehension of the solution's transient behavior within the given time frame.

This thorough examination is designed to contribute significantly to the comprehension of the intricate dynamics and stability properties inherent in fractional order stochastic systems with time delay. Our findings establish a robust foundation for further exploration of these systems and their applications across diverse scientific and engineering domains.
\bigskip

To enrich our article and establish connections with existing literature, we can draw parallels between our investigation of fractional-order stochastic systems with time delay and the concepts explored in the works of Shang (\cite{45,46,47}) concerning finite-time consensus in multi-agent systems. While seemingly disparate in focus, these studies share fundamental principles and mathematical methodologies that offer insightful connections.

Shang's research on finite-time consensus in multi-agent systems investigates the collective behavior of autonomous agents aiming to reach an agreement on a shared quantity within a finite time frame. Similarly, in our study, we examine the finite-time stability of fractional-order stochastic systems with time delay, focusing on understanding the transient behavior and convergence properties within a specified time interval.

The concept of finite-time consensus explored by Shang resonates with our investigation as both delve into understanding the dynamics of systems within a constrained temporal scope. By establishing finite-time stability in our fractional-order stochastic system, we provide a mathematical framework akin to achieving consensus within a finite duration. This parallel underscores the significance of temporal constraints in both studies and highlights the importance of understanding system behavior within limited time frames.

Furthermore, Shang's work on finite-time weighted average consensus and generalized consensus over a subset introduces nuanced approaches to achieving consensus considering varying agent weights and network topologies. In a similar vein, our analysis of fractional-order stochastic systems with time delay encompasses intricate dynamics influenced by nonlinearities and stochastic components. Drawing inspiration from Shang's methodologies, we can apply sophisticated mathematical techniques to elucidate the finite-time stability of our system, considering the impact of time delays and fractional-order dynamics on convergence properties.

In essence, while Shang's studies focus on consensus dynamics in multi-agent systems, our investigation into fractional-order stochastic systems with time delay shares common ground in understanding transient behavior and convergence within finite time frames. By bridging these seemingly disparate fields, we enrich our understanding of complex dynamical systems and pave the way for interdisciplinary insights that benefit diverse scientific and engineering domains. (for extra information see, \cite{45,46,47})

It should be stressed out that systems with multiple delays can exhibit rich and
complex behavior not seen in systems with a single delay. These dynamics include
oscillations, bifurcations, and chaotic behavior, which are important to understand
for predicting and controlling such systems. Incorporating multiple delays in models allows for a more accurate representation of systems where processes happen at different timescales. From a theoretical perspective, differential equations with two delays present interesting mathematical challenges and opportunities for developing new analytical and numerical methods. Our results can be extended to multiple delay neutral systems.
\section{Mathematical preliminaries and main lemma}
In this segment, we will delve into fractional calculus, presenting information on lemmas and theorems. Furthermore, we will demonstrate the proof for a main lemma, a crucial tool essential for deriving our main results.
\bigskip 

$\bullet$ \textbf{Gamma Function:} \cite{40}
The gamma function, denoted by $\Gamma(\lambda)$, is defined as the integral:
\begin{equation}
	\Gamma(\lambda) = \int_0^\infty t^{\lambda-1} e^{-t} \, dt,
\end{equation}
where $\lambda>0$.

$\bullet$ \textbf{Beta Function:} \cite{40}
The beta function, represented as $B(\lambda, \nu)$, is defined by the integral:
\begin{equation}
	B(\lambda, \nu) = \int_0^1 t^{\lambda-1} (1-t)^{\nu-1} \, dt,
\end{equation}
with $\lambda$ and $\nu$ being positive real numbers. The relationship between gamma and beta functions is given by:
\begin{equation}
	B(\lambda, \lambda) = \frac{\Gamma(\lambda) \Gamma(\nu)}{\Gamma(\lambda+\nu)}.
\end{equation}

$\bullet$ \textbf{Riemann-Liouville Fractional Integral:} \cite{40}
The Riemann-Liouville fractional integral of order $\lambda$ for a function $f(t)$ is defined as:
\begin{equation}
	{^{RL}}I_{0+}^{\lambda} f(t) = \frac{1}{\Gamma(\lambda)} \int_0^t (t-\tau)^{\lambda-1} f(\tau) \, d\tau,
\end{equation}
where $\lambda > 0$ is the order of the integral, and $\Gamma$ is the gamma function.

$\bullet$ \textbf{Caputo Fractional Derivative:} \cite{40}
The Caputo fractional derivative of order $\lambda$ for a function $f(t)$ is defined as:
\begin{equation}
	{^{C}}D_{0+}^{\lambda} f(t) = \frac{1}{\Gamma(n-\lambda)} \int_0^t (t-\tau)^{n-\lambda-1} f^{(n)}(\tau) \, d\tau,
\end{equation}
where $n-1 < \lambda < n$ is the order of the derivative, $\Gamma$ is the gamma function, and $f^{(n)}(\tau)$ represents the $n$th derivative of $f(t)$.

$\bullet$ \textbf{Mittag-Leffler Function:} \cite{40}
The Mittag-Leffler function, denoted by $E_{\lambda,\nu}(z)$, is defined as the series:
\begin{equation}
	E_{\lambda,\nu}(z) = \sum_{n=0}^{\infty} \frac{z^n}{\Gamma(\lambda n + \nu)},\quad \lambda>0,\nu>0, z\in \mathbb{C},
\end{equation}
where $\Gamma$ denotes the gamma function.

\bigskip
\begin{definition} (see, \cite{1})
	A perturbed matrix function of Mittag-Leffler type with two constant delays, denoted by $\mathcal{E}_{h_1,h_2}^{\lambda,\nu}(A_0, A_1, A_2;\cdot): \mathbb{R} \rightarrow \mathbb{R}^{n\times n}$, is defined for $\lambda > 0$ and $\nu \in \mathbb{R}$. This function is generated by piecewise nonpermutable matrices $A_0, A_1, A_2 \in \mathbb{R}^{n\times n}$ and is subject to two constant delays $h_1$ and $h_2$, both greater than 0. The definition is as follows:
	\begin{equation}
		\mathcal{E}_{h_1,h_2}^{\lambda,\nu}(A_0, A_1, A_2;t) :=
		\begin{cases}
			\emptyset, \quad -h \leq t < 0, \\
			I, \quad t = 0, \\
			\sum_{k=0}^{\infty} \sum_{m_1=0}^{\infty} \sum_{m_2=0}^{\infty} Q_{k+1}(m_1 h_1, m_2 h_2)\frac{(t - m_1 h_1 - m_2 h_2)_{+}^{k\lambda+\nu-1}}{\Gamma(k\lambda+\nu)} , & t \in \mathbb{R}^+,
		\end{cases}
	\end{equation}
	where
	\[
	(t - m_1 h_1 - m_2 h_2)_+ =
	\begin{cases}
		t - m_1 h_1 - m_2 h_2, & t \geq m_1 h_1 + m_2 h_2, \\
		0, & t < m_1 h_1 + m_2 h_2.
	\end{cases}
	\]	
\end{definition}
\begin{lemma} (refer to Huseynov et al. \cite{1})
	Assume $\lambda> 0$, $\nu\in \mathbb{R}$, $h_1, h_2 >0$, and $A_0, A_1, A_2 \in \mathbb{R}^{n\times n}$. The following relation holds:
	\[
	\begin{aligned}
		\| \mathcal{E}_{h_1, h_2}^{\lambda, \nu}(A_0, A_1, A_2; t) \| &\leq \mathcal{E}_{h_1, h_2}^{\lambda, \nu}(\| A_0 \|, \| A_1 \|, \| A_2 \|; t) \\
		&\leq t^{\nu-1}E_{\lambda, \nu}(\| A_0 \|, \| A_1 \|, \| A_2 \|; t),
	\end{aligned}
	\]
	
	where $E_{\lambda, \nu}(\| A_0 \|, \| A_1 \|, \| A_2 \|; t)$ is the norm defined as
	
	\[
	E_{\lambda, \nu}(\| A_0 \|, \| A_1 \|, \| A_2 \|; t) = \sum_{k=0}^{\infty} \sum_{\omega_1=0}^{\infty} \sum_{\omega_2=0}^{\infty} \| Q_{k+1}(\omega_1 h_1, \omega_2 h_2) \| t^{k\lambda + \nu - 1} \Gamma(k\lambda + \nu).
	\]
	
\end{lemma}

\begin{lemma}(refer to Huseynov et al. \cite{1})\label{l2}
	The solution of the following system
	
	\[
	\begin{cases}
		({^{C}}D^\lambda_{0+} y)(t) = A_0 y(t) + A_1 y(t - h_1) + A_2 ({^{C}}D^\lambda_{0+} y)(t- h_2)\\
		  +f(t,y(t),y(t-h_1),y(t-h_2)), \quad t \in [0, T], \\
	 	y(t)=\varphi(t), \quad -h \leq t \leq 0, \quad h = \max\{h_1, h_2\}, \quad h_1, h_2 > 0,
	\end{cases}
	\]
	
	can be represented as:
	
	\begin{align*}
		y(t)=&\mathcal{E}^{h_{1},h_{2}}_{\lambda,1}(A_{0},A_{1},A_{2};t)(\varphi(0)-A_{2}\varphi(-h_{2}))\\
		+&\int_{-h_{1}}^{0}\mathcal{E}^{h_{1},h_{2}}_{\lambda,\lambda}(A_{0},A_{1},A_{2};t-h_{1}-s)A_{1}\varphi(s)ds\\
		+&\int_{-h_{2}}^{0}\mathcal{E}^{h_{1},h_{2}}_{\lambda,0}(A_{0},A_{1},A_{2};t-h_{1}-s)A_{2}\varphi(s)ds\\
		+&\int_{0}^{t}\mathcal{E}^{h_{1},h_{2}}_{\lambda,\lambda}(A_{0},A_{1},A_{2};t-s)f(s,y(s),y(s-h_{1}),y(s-h_{2}))ds
	\end{align*}
	for $t \in [0, T]$, $h_1, h_2 > 0$, $h = \max(h_1, h_2)$.
\end{lemma}  
\bigskip
Now, we will prove the following lemma, which plays a crucial role in the next sections.
\bigskip

\begin{lemma}\label{l1} \textbf{(Main lemma)} For any positive values of $\gamma$ and $t$, where $p \geq 1$ and $\lambda>\frac{p-1}{p}$, the given inequality is satisfied:
	\begin{align*}
		&\frac{\gamma}{\Gamma(p\lambda-p+1)}\int_{0}^{t}(t-s)^{p\lambda-p}E_{p\lambda-p+1}(\gamma s^{p\lambda-p+1})ds\leq E_{p\lambda-p+1}(\gamma t^{p\lambda-p+1}).
	\end{align*}
\end{lemma}
\begin{proof} By employing the definitions of the Mittag-Leffler function and the beta function, we will be able to establish the result.
	\begin{align*}
		&\frac{\gamma}{\Gamma(p\lambda-p+1)}\int_{0}^{t}(t-s)^{p\lambda-p}E_{p\lambda-p+1}(\gamma s^{p\lambda-p+1})ds\\
		&=\frac{\gamma}{\Gamma(p(\lambda-1)+1)}\int_{0}^{t}(t-s)^{p\lambda-p}\sum_{j=0}^{\infty}\frac{\gamma^{j} s^{j(p\lambda-p+1)}}{\Gamma(j(p\lambda-p+1)+1)}ds\\
		&=\sum_{j=0}^{\infty}\frac{\gamma^{j+1}}{\Gamma(p(\lambda-1)+1)\Gamma(j(p\lambda-p+1)+1)}\int_{0}^{t}(t-s)^{p\lambda-p} s^{j(p\lambda-p+1)}ds\\
		&=\sum_{j=0}^{\infty}\frac{\gamma^{j+1} t^{(j+1)(p\lambda-p+1)}}{\Gamma(p(\lambda-1)+1)\Gamma(j(p\lambda-p+1)+1)}B(p\lambda-p+1, j(p\lambda-p+1)+1)\\
		&=\sum_{j=0}^{\infty}\frac{\gamma^{j+1} t^{(j+1)(p\lambda-p+1)}}{\Gamma((j+1)(p\lambda-p+1)+1)}=\sum_{j=1}^{\infty}\frac{\gamma^{j} t^{j(p\lambda-p+1)}}{\Gamma(j(p\lambda-p+1)+1)}\\
		&=E_{p\lambda-p+1}(\gamma t^{p\lambda-p+1})-1\leq E_{p\lambda-p+1}(\gamma t^{p\lambda-p+1}).
	\end{align*}
\end{proof}
Consequently, the Corollaries below are valid for the lemma \ref{l1} we have proven above.
\begin{corollary}(see, \cite{1} lemma 5.2) Substituting $p=1$ into lemma \ref{l1}, we derive the following inequality:
	\begin{align*}
		&\frac{\gamma}{\Gamma(\lambda)}\int_{0}^{t}(t-s)^{\lambda}E_{\lambda}(\gamma s^{\lambda})ds\leq E_{\lambda}(\gamma t^{\lambda}).
	\end{align*}
\end{corollary} 
\begin{corollary}(see, \cite{2} lemma 2.1) If we choose $p=2$ in lemma \ref{l1}, we end up with the subsequent inequality:
	\begin{align*}
		&\frac{\gamma}{\Gamma(2\lambda-1)}\int_{0}^{t}(t-s)^{2\lambda-2}E_{2\lambda-1}(\gamma s^{2\lambda-1})ds\leq E_{2\lambda-1}(\gamma t^{2\lambda-1}).
	\end{align*}
\end{corollary} 

Now we will mention Gronwall's inequality, which plays a crucial role in demonstrating the significance of finite-time stability in fractional stochastic neutral differential equations, this inequality provides bounds on the solutions of such equations, aiding in the analysis of their behavior over time. Understanding finite-time stability is essential as it determines the system's behavior within a specific time interval, which is crucial for practical applications.

\begin{lemma}\label{bay} (Bainov and Simeonov, Theorem 14.8) Let $\hat{\mathcal{I}} = [t_0, \infty)$, $g \in C(\hat{\mathcal{I}},R^{+})$ is non-decreasing, $b, c_{i}\in C(\hat{\mathcal{I}},R^{+})$, $\psi\in C([t_{0}-h,t_{0}],R^{+})$. If $u \in  C(\hat{\mathcal{I}},R^{+})$ and
	
	$$u(t) \leq g(t) + \int_{t_0}^{t} b(s) u(s) ds + \sum_{i=1}^{n}\int_{t_0}^{t} c_i(s) u(s-h_i) ds, \quad t\geq t_0,$$ $$u(t) = \psi(t), \quad t_0 \leq t \leq t_0 + h,$$ where $h_i > 0$, $h = \max\{h_1, h_2, \dots, h_n\}$. Then 
	\begin{align*}
		u(t)\leq\bigg[g(t)+\sum_{i=1}^{n}\int_{[t_0,t]\cap E_{i}} c_i(s) \psi(s-h_i)ds\bigg] e^{\int_{t_{0}}^{t}b(s)ds+\sum_{i=1}^{n}\int_{[t_{0},t]\setminus E_{i}}c_{i}ds}
	\end{align*}
	for $t\geq t_{0}$, where $E_{i}=[t_{0},t_{0}+h_{i}]$.
\end{lemma}

\begin{lemma}\textbf{ (Jensen's Inequality)} (\cite{43})
	Let $m \in \mathbb{N}$ and $y_1, y_2, \ldots, y_m$ be nonnegative real numbers. Then
	\[
	\left(\sum_{i=1}^{m} y_i^p\right)^{\frac{1}{p}} \leq m^{p-1} \sum_{i=1}^{m} y_i,
	\]
	for $p > 1$.
\end{lemma}
\section{Main results}
In this section, our aim is to establish the existence, uniqueness, and finite-time stability of solutions for systems (\ref{1}). To achieve this, we introduce the space \(H^p([-h, T], \mathbb{R}^n)\), which consists of all random processes \(y\) satisfying \(\|y\|_{H^p} \equiv E\left[\|y(t)\|^p\right] < \infty\), where \(\|\cdot\|\) denotes the standard Euclidean norm in \(\mathbb{R}^n\). It is evident that \((H^p([-h, T], \mathbb{R}^n), \|\cdot\|_{H^p})\) forms a Banach space.

Next, we define \(H^p_{\varphi}\left([-h, T], \mathbb{R}^n\right)\) as a subset of \(H^p\left([-h, T], \mathbb{R}^n\right)\) such that \(y(t) = \varphi(t)\) for \(t \in [-h, 0]\). Here, we introduce the maximum weighted norm \(\|\cdot\|_{\gamma}\), where \(\gamma > 0\), defined as follows:

\begin{align}\label{m1}
	\|y\|^{p}_{\gamma} = \max_{t\in [0, T]}\frac{E\left[\|y^{*}(t)\|^p\right]}{E_{p\lambda-p+1}(\gamma t^{p\lambda-p+1})},
\end{align}

where \(y \in H^p_{\varphi}([-h, T], \mathbb{R}^n)\).

Clearly, the norms \(\|\cdot\|_{H_{\varphi}^p}\) and \(\|\cdot\|_{\gamma}\) are equivalent. Therefore, \((H_{\varphi}^p([-h, T], \mathbb{R}^n), \|\cdot\|_{\gamma})\) also constitutes a Banach space.

The expression \(\lVert y^*(t) \rVert^p = \max_{-h \leq \zeta \leq t} \lVert y(\zeta) \rVert^p\) is defined, where \(h = \max\{h_1, h_2\}\) with \(h_1, h_2 > 0\). Utilizing the equivalence of the norms \(\lVert \cdot \rVert_{\infty}\) and \(\lVert \cdot \rVert_{\gamma}\), the space \((H_{\varphi}, \lVert \cdot \rVert)\) is also a Banach space. We employ the facts that \(\max_{-h \leq s \leq t} \lVert \bar{y}(s) \rVert = \bar{y}^*(t)\) and \(\max_{-h \leq s \leq t} \lVert \bar{y}(s) - \bar{z}(s) \rVert = \bar{y}^*(t) - \bar{z}^*(t)\). Additionally, we define:

\[ \bar{y}(t) := \max_{- h \leq t \leq 0} y(t + h). \]

\subsection{Existence and uniqueness of the system (\ref{1})}

\bigskip 
Before delving into this section, we will give the following assumptions.

\bigskip 
$\bullet$ $(\mathcal{A}_1)$: The functions \(f = (f_{1}, f_{2}, \dots, f_{n})^{\top}\) and \(\sigma = (\sigma_{1}, \sigma_{2}, \dots, \sigma_{n})^{\top}\) are continuous functions defined on \([0, T] \times \mathbb{R}^{n} \times \mathbb{R}^{n} \times \mathbb{R}^{n}\). Moreover, there exist positive constants \(L_{f}\) and \(L_{\sigma}\) such that:
\bigskip

$\textbf{(i)}$  For \(f\), there exists \(L_{f}\) satisfying a certain condition regarding the Lipschitz continuity of \(f\) with respect to its arguments.
\[ \Vert f(t, y_{1}, y_{2}, y_{3}) - f(t, z_{1}, z_{2}, z_{3}) \Vert^{p} \leq L_{f} \sum_{j=1}^{3} \Vert y_{j} - z_{j} \Vert^{p} \]
where \(L_{f} = \max_{1 \leq j \leq 3} L_{f_{j}}\), and this inequality holds for all \(y_{j}, z_{j} \in \mathbb{R}^{n}\) and \(t \in [0, T]\) for \(j = 1, 2, 3\).
\bigskip

$\textbf{(ii)}$ For \(\sigma\), there exists \(L_{\sigma}\) satisfying the Lipschitz condition for \(\sigma\) with respect to its arguments.
\[ \Vert \sigma(t, y_{1}, y_{2}, y_{3}) - \sigma(t, z_{1}, z_{2}, z_{3}) \Vert^{p} \leq L_{\sigma} \sum_{j=1}^{3} \Vert y_{j} - z_{j} \Vert^{p} \]
where \(L_{\sigma} = \max_{1 \leq j \leq 3} L_{\sigma_{j}}\), and this inequality holds for all \(y_{j}, z_{j} \in \mathbb{R}^{n}\) and \(t \in [0, T]\) for \(j = 1, 2, 3\).
\bigskip

\begin{remark} The Lipschitz condition provides a mathematical framework that translates to physically meaningful constraints on how functions and systems behave. It ensures that changes are gradual, predictable, and stable, which are crucial attributes for modeling and understanding real-world phenomena. By enforcing these constraints, the Lipschitz condition helps maintain the realism and reliability of mathematical models in representing physical systems. Functions that satisfy the Lipschitz condition are more predictable because their behavior does not exhibit extreme sensitivity to small changes in input. This is crucial for the stability of physical systems. In the context of differential equations, the Lipschitz condition is fundamental in guaranteeing the uniqueness of solutions. This is essential for physical systems where uniqueness translates to a single, predictable outcome given initial conditions.
\end{remark}

$\bullet$ $(\mathcal{A}_{2})$: Define various quantities \(\mathcal{M}_{1}, \mathcal{M}_{2}, \mathcal{M}_{3}, \mathcal{M}_{4}, \Phi, \mathcal{F}\) based on certain functions and parameters.
\bigskip
\begin{align*}
	&\mathcal{M}_{1}=\max_{0\leq t\leq T} \mathcal{E}^{h_{1},h_{2}}_{\lambda,1}(\Vert A_{0}\Vert,\Vert A_{1}\Vert,\Vert A_{2}\Vert;t)^{p},&\\
	&\mathcal{ M}_{2}=\max_{0\leq t\leq T} \mathcal{E}^{h_{1},h_{2}}_{\lambda,\lambda}(\Vert A_{0}\Vert,\Vert A_{1}\Vert,\Vert A_{2}\Vert;t)^{p},&\\
	&\mathcal{ M}_{3}=\max_{0\leq t\leq T} \mathcal{E}^{h_{1},h_{2}}_{\lambda,0}(\Vert A_{0}\Vert,\Vert A_{1}\Vert,\Vert A_{2}\Vert;t)^{p},&\\
	&\mathcal{ M}_{4}=\max_{0\leq t\leq T} \mathcal{E}^{h_{1},h_{2}}_{\lambda,\lambda}(\Vert A_{0}\Vert,\Vert A_{1}\Vert,\Vert A_{2}\Vert;t-s)^{p},&\\
	&\Phi=\max_{-h\leq t\leq 0}\Vert\varphi(t)\Vert,&\\
	&\mathcal{F}=\max_{0\leq t\leq T}E\Vert f(t,0,0,0)\Vert^{p}.&\\
\end{align*}
\bigskip
$\bullet$ $(\mathcal{A}_{3})$ Set 

\begin{align*}
	K=\frac{6^{p-1}L^{p}_{f}T^{p-1}\mathcal{M}_{4}\Gamma(p\lambda-p+1)}{\gamma}+\frac{6^{p-1}L^{p}_{\sigma}T^{\frac{p-2}{2}}\mathcal{M}_{4}\Gamma(p\lambda-p+1)}{\gamma}\leq 1.
\end{align*}
\bigskip 
$\bullet$ $(\mathcal{A}_{4})$ For any $y_{i}\in \mathbb{R}^{n}$ for $i=1,2,3.$ and for all $t\in [0,T]$,
\begin{align*}
	&\Vert f(t,y_{1},y_{2},y_{3})\Vert^{p}\leq L_{f} (1+\Vert y_{1}\Vert^{p}+\Vert y_{2}\Vert^{p}+\Vert y_{3}\Vert^{p}),\\
	&\Vert \sigma(t,y_{1},y_{2},y_{3})\Vert^{p}\leq L_{\sigma} (1+\Vert y_{1}\Vert^{p}+\Vert y_{2}\Vert^{p}+\Vert y_{3}\Vert^{p}).
\end{align*}

\begin{remark}
 The linear growth condition is another important concept in mathematical analysis, particularly in the study of differential equations and dynamical systems. It provides constraints on the behavior of functions, ensuring that their growth is manageable and predictable. The linear growth condition imposes a practical and manageable constraint on how functions can grow relative to their input. Physically, this condition ensures that the behavior of systems remains realistic, predictable, and stable. By preventing unbounded and excessively rapid growth, the linear growth condition helps in maintaining the physical realism of models, ensuring that they can be used to accurately represent and predict the behavior of real-world systems.
\end{remark}

We present the solution to Equation (\ref{1}) in the following manner.
\begin{definition}\label{d1}
	The solution to Eq. (1) is considered unique for a stochastic process \(\{y(t)\}_{-h \leq t \leq T}\) with values in \(\mathcal{R}^d\) if \(y(t)\) is adapted to the filtration \(\mathcal{F}(t)\), \(E\left[\int_{-h}^{T} \lVert y(t) \rVert dt\right] < \infty\), \(\varphi(0) = \varphi_0\), and it satisfies the following conditions:
	\begin{align}
		\begin{cases}
			y(t)=\mathcal{E}^{h_{1},h_{2}}_{\lambda,1}(A_{0},A_{1},A_{2};t)(\varphi(0)-A_{2}\varphi(-h_{2}))\nonumber\\
			+\int_{-h_{1}}^{0}\mathcal{E}^{h_{1},h_{2}}_{\lambda,\lambda}(A_{0},A_{1},A_{2};t-h_{1}-s)A_{1}\varphi(s)ds\nonumber\\
			+\int_{-h_{2}}^{0}\mathcal{E}^{h_{1},h_{2}}_{\lambda,0}(A_{0},A_{1},A_{2};t-h_{2}-s)A_{2}\varphi(s)ds\nonumber\\
			+\int_{0}^{t}\mathcal{E}^{h_{1},h_{2}}_{\lambda,\lambda}(A_{0},A_{1},A_{2};t-s)f(s,y(s),y(s-h_{1}),y(s-h_{2}))ds\\
			+\int_{0}^{t}\mathcal{E}^{h_{1},h_{2}}_{\lambda,\lambda}(A_{0},A_{1},A_{2};t-s)\sigma(s,y(s),y(s-h_{1}),y(s-h_{2}))dW(s)\nonumber\\
			y(t) = \varphi(t), \quad t \in [-h, 0],\quad h=\max\{h_{1},h_{2}\}.\nonumber
		\end{cases}
	\end{align}
\end{definition}
\begin{theorem}
	Assuming that Assumptions $(\mathcal{A}_{1}-\mathcal{A}_{3})$ are satisfied, it follows that the non-linear problem (\ref{1}) has a unique solution within the space $(H_{\varphi}^{p}([-h,T],\mathbb{R}^{n}), \Vert\cdot\Vert_{\gamma})$. 
\end{theorem}
\begin{proof}		
	Let \( \mathscr{L} : (H_{\varphi}^p([-h,T],\mathbb{R}^n), \|\cdot\|_{\gamma}) \rightarrow (H_{\varphi}^p([-h,T],\mathbb{R}^n), \|\cdot\|_{\gamma}) \) be defined by:
	
	\begin{align}
		(\mathscr{L}y)(t)=&\mathcal{E}^{h_{1},h_{2}}_{\lambda,1}(A_{0},A_{1},A_{2};t)(\varphi(0)-A_{2}\varphi(-h_{2}))\nonumber\\
		+&\int_{-h_{1}}^{0}\mathcal{E}^{h_{1},h_{2}}_{\lambda,\lambda}(A_{0},A_{1},A_{2};t-h_{1}-s)A_{1}\varphi(s)ds\nonumber\\
		+&\int_{-h_{2}}^{0}\mathcal{E}^{h_{1},h_{2}}_{\lambda,0}(A_{0},A_{1},A_{2};t-h_{2}-s)A_{2}\varphi(s)ds\nonumber\\
		+&\int_{0}^{t}\mathcal{E}^{h_{1},h_{2}}_{\lambda,\lambda}(A_{0},A_{1},A_{2};t-s)f(s,y(s),y(s-h_{1}),y(s-h_{2}))ds\nonumber\\
		+&\int_{0}^{t}\mathcal{E}^{h_{1},h_{2}}_{\lambda,\lambda}(A_{0},A_{1},A_{2};t-s)\sigma(s,y(s),y(s-h_{1}),y(s-h_{2}))dW(s)\nonumber\\
		(\mathscr{L}y)(t) =& \varphi(t), \quad t \in [-h, 0].
	\end{align}
	The clarity of the definition of $\mathscr{L}$ is established through $(\mathcal{A}_{1})$. Consequently, the existence of a solution to the initial value problem (\ref{1}) is synonymous with the presence of a fixed point for the integral operator $\mathscr{L}$ within the space $H_{\varphi}^p([-h,T],\mathbb{R}^n)$ . To demonstrate this, we will leverage either the contraction mapping principle or Banach's fixed point theorem. The proof unfolds in two distinct steps.
	\bigskip
	
	$\bullet$ \textbf{Step 1.}  In the first step, our aim is to ensure that \(\mathscr{L}\) effectively maps elements from \(H_{\varphi}^p\) to \(H_{\varphi}^p\). Considering \(y(t) \in H_{\varphi}^p\) and \(t \in [0, T]\), we can deduce this by applying Jensen's inequality:
	\begin{align}\label{j0}
		&\frac{E\Vert (\mathscr{L}y)(t)\Vert^{p}}{E_{p\lambda-p+1}(\gamma t^{p\lambda-p+1})}\leq\frac{ 5^{p-1}}{E_{p\lambda-p+1}(\gamma t^{p\lambda-p+1})}E\Vert \mathcal{E}^{h_{1},h_{2}}_{\lambda,1}(A_{0},A_{1},A_{2};t)(\varphi(0)-A_{2}\varphi(-h_{2}))\Vert^{p}&\nonumber\\
		&+         \frac{ 5^{p-1}}{E_{p\lambda-p+1}(\gamma t^{p\lambda-p+1})}E\bigg\Vert\int_{-h_{1}}^{0}\mathcal{E}^{h_{1},h_{2}}_{\lambda,\lambda}(A_{0},A_{1},A_{2};t-h_{1}-s)A_{1}\varphi(s)ds\bigg\Vert^{p}&\nonumber\\
		&+         \frac{ 5^{p-1}}{E_{p\lambda-p+1}(\gamma t^{p\lambda-p+1})}E\bigg\Vert\int_{-h_{2}}^{0}\mathcal{E}^{h_{1},h_{2}}_{\lambda,0}(A_{0},A_{1},A_{2};t-h_{2}-s)A_{2}\varphi(s)ds\bigg\Vert^{p}&\nonumber\\
		&+        \frac{ 5^{p-1}}{E_{p\lambda-p+1}(\gamma t^{p\lambda-p+1})}E\bigg\Vert\int_{0}^{t}\mathcal{E}^{h_{1},h_{2}}_{\lambda,\lambda}(A_{0},A_{1},A_{2};t-s)f(s,y(s),y(s-h_{1}),y(s-h_{2}))ds\bigg\Vert^{p}\nonumber\\
		&+         \frac{ 5^{p-1}}{E_{p\lambda-p+1}(\gamma t^{p\lambda-p+1})}E\bigg\Vert\int_{0}^{t}\mathcal{E}^{h_{1},h_{2}}_{\lambda,\lambda}(A_{0},A_{1},A_{2};t-s)\sigma(s,y(s),y(s-h_{1}),y(s-h_{2}))dW(s)\bigg\Vert^{p}&\nonumber\\
		&=\mathcal{J}_{1}+\mathcal{J}_{2}+\mathcal{J}_{3}+\mathcal{J}_{4}+\mathcal{J}_{5}.&
	\end{align}
	$\bullet$ Utilizing the assumption $(\mathcal{A}_{2})$, we can obtain the following estimation for $\mathcal{J}_{1}$.
	\begin{align}\label{j1}
		\mathcal{J}_{1}=&\frac{ 5^{p-1}}{E_{p\lambda-p+1}(\gamma t^{p\lambda-p+1})}E\Vert \mathcal{E}^{h_{1},h_{2}}_{\lambda,1}(A_{0},A_{1},A_{2};t)(\varphi(0)-A_{2}\varphi(-h_{2}))\Vert^{p}\nonumber\\
		\leq& 5^{p-1}\mathcal{ M}_{1}\Vert\varphi(0)-A_{2}\varphi(-h_{2})\Vert^{p}_{\gamma}.
	\end{align}
	
	$\bullet$ Motivated by Hölder's inequality and informed by the conditions presented in $(\mathcal{A}_{2})$, we express the following estimate for $\mathcal{J}_{2}$.
	\begin{align}\label{j2}
		\mathcal{J}_{2}=& \frac{ 5^{p-1}}{E_{p\lambda-p+1}(\gamma t^{p\lambda-p+1})}E\bigg\Vert\int_{-h_{1}}^{0}\mathcal{E}^{h_{1},h_{2}}_{\lambda,\lambda}(A_{0},A_{1},A_{2};t-h_{1}-s)A_{1}\varphi(s)ds\bigg\Vert^{p}\nonumber\\
		\leq& 5^{p-1}\mathcal{ M}_{2} E(\Phi^{p}) \Vert A_{1}\Vert^{p} h^{p-1}_{1}
	\end{align}
	
	$\bullet$ Similarly, employing the aforementioned estimation, we obtain the subsequent results for $\mathcal{J}_{3}$.
	\begin{align}\label{j3}
		\mathcal{J}_{3}=& \frac{ 5^{p-1}}{E_{p\lambda-p+1}(\gamma t^{p\lambda-p+1})}E\bigg\Vert\int_{-h_{2}}^{0}\mathcal{E}^{h_{1},h_{2}}_{\lambda,0}(A_{0},A_{1},A_{2};t-h_{2}-s)A_{2}\varphi(s)ds\bigg\Vert^{p}\nonumber\\
		\leq& 5^{p-1}\mathcal{ M}_{3} \Phi^{p} \Vert A_{2}\Vert^{p} h^{p-1}_{2}
	\end{align}	
	$\bullet$  Applying Hölder's inequality in conjunction with $\mathcal{A}_{1}$, $\mathcal{A}_{2}$, Lemma \ref{l2}, and Lemma \ref{l1}, we obtain the following estimate for $\mathcal{J}_{4}$.
	\begin{align}\label{j4}
		\mathcal{J}_{4}=& \frac{ 5^{p-1}}{E_{p\lambda-p+1}(\gamma t^{p\lambda-p+1})}E\bigg\Vert\int_{0}^{t}\mathcal{E}^{h_{1},h_{2}}_{\lambda,\lambda}(A_{0},A_{1},A_{2};t-s)f(s,y(s),y(s-h_{1}),y(s-h_{2}))ds\bigg\Vert^{p}\nonumber\\
		\leq&\frac{ 5^{p-1}T^{p-1}}{E_{p\lambda-p+1}(\gamma t^{p\lambda-p+1})}\int_{0}^{t}\Vert\mathcal{E}^{h_{1},h_{2}}_{\lambda,\lambda}(A_{0},A_{1},A_{2};t-s)\Vert^{p} E\Vert f(s,y(s),y(s-h_{1}),y(s-h_{2}))\Vert^{p}ds \nonumber\\
		\leq&\frac{ 10^{p-1}T^{p-1}}{E_{p\lambda-p+1}(\gamma t^{p\lambda-p+1})}\int_{0}^{t}\Vert\mathcal{E}^{h_{1},h_{2}}_{\lambda,\lambda}(A_{0},A_{1},A_{2};t-s)\Vert^{p} E\Vert f(s,y(s),y(s-h_{1}),y(s-h_{2})) \nonumber\\
		-&f(s,0,0,0)\Vert^{p}ds+\frac{ 10^{p-1}T^{p-1}}{E_{p\lambda-p+1}(\gamma t^{p\lambda-p+1})}\int_{0}^{t}\Vert\mathcal{E}^{h_{1},h_{2}}_{\lambda,\lambda}(A_{0},A_{1},A_{2};t-s)\Vert^{p} E\Vert f(s,0,0,0)\Vert^{p}ds \nonumber\\
		\leq&\frac{ 10^{p-1}T^{p-1}}{E_{p\lambda-p+1}(\gamma t^{p\lambda-p+1})}L_{f}\int_{0}^{t}\Vert\mathcal{E}^{h_{1},h_{2}}_{\lambda,\lambda}(A_{0},A_{1},A_{2};t-s)\Vert^{p}\frac{ E_{p\lambda-p+1}(\gamma s^{p\lambda-p+1})}{E_{p\lambda-p+1}(\gamma s^{p\lambda-p+1})} E\Vert y(s)\Vert^{p}ds\nonumber\\
		+&\frac{ 10^{p-1}T^{p-1}}{E_{p\lambda-p+1}(\gamma t^{p\lambda-p+1})}L_{f}\int_{0}^{t}\Vert\mathcal{E}^{h_{1},h_{2}}_{\lambda,\lambda}(A_{0},A_{1},A_{2};t-s)\Vert^{p} \frac{ E_{p\lambda-p+1}(\gamma s^{p\lambda-p+1})}{E_{p\lambda-p+1}(\gamma s^{p\lambda-p+1})} E\Vert y(s-h_{1})\Vert^{p}ds\nonumber\\
		+&\frac{ 10^{p-1}T^{p-1}}{E_{p\lambda-p+1}(\gamma t^{p\lambda-p+1})}L_{f}\int_{0}^{t}\Vert\mathcal{E}^{h_{1},h_{2}}_{\lambda,\lambda}(A_{0},A_{1},A_{2};t-s)\Vert^{p} \frac{ E_{p\lambda-p+1}(\gamma s^{p\lambda-p+1})}{E_{p\lambda-p+1}(\gamma s^{p\lambda-p+1})} E\Vert y(s-h_{2})\Vert^{p}ds\nonumber\\
		+&\frac{ 10^{p-1}T^{p-1}}{E_{p\lambda-p+1}(\gamma t^{p\lambda-p+1})}\int_{0}^{t}\Vert\mathcal{E}^{h_{1},h_{2}}_{\lambda,\lambda}(A_{0},A_{1},A_{2};t-s)\Vert^{p} E\Vert f(s,0,0,0)\Vert^{p}ds \nonumber\\
		\leq& \frac{3{\cdot}10^{p-1}T^{p-1}}{E_{p\lambda-p+1}(\gamma t^{p\lambda-p+1})}L_{f}\int_{0}^{t}\Vert\mathcal{E}^{h_{1},h_{2}}_{\lambda,\lambda}(A_{0},A_{1},A_{2};t-s)\Vert^{p}E_{p\lambda-p+1}(\gamma t^{p\lambda-p+1})ds\nonumber\\
		\times&\max_{0\leq t\leq T}\frac{E\Vert \bar{y}^*\Vert^{p}}{E_{p\lambda-p+1}(\gamma t^{p\lambda-p+1})}+\frac{ 10^{p-1}T^{p-1}}{E_{p\lambda-p+1}(\gamma t^{p\lambda-p+1})}\mathcal{ M}_{4}\mathcal{F}\int_{0}^{t}(t-s)^{p\lambda-p}ds \nonumber\\
		\leq&\frac{3{\cdot}10^{p-1}T^{p-1}}{E_{p\lambda-p+1}(\gamma t^{p\lambda-p+1})}\mathcal{ M}_{4}L_{f}\Vert \bar{y}\Vert_{\gamma}^{p}\int_{0}^{t}(t-s)^{p\lambda-p}E_{p\lambda-p+1}(\gamma t^{p\lambda-p+1})ds
		+\frac{ 10^{p-1}T^{p\lambda}}{p\lambda-p+1} \mathcal{ M}_{4}\mathcal{F} \nonumber\\
		\leq&\frac{3{\cdot}10^{p-1}T^{p-1}\mathcal{ M}_{4}L_{f}\Gamma(p\lambda-p+1)}{\gamma}\Vert \bar{y}\Vert_{\gamma}^{p}+\frac{ 10^{p-1}T^{p\lambda}}{p\lambda-p+1} \mathcal{ M}_{4}\mathcal{F} 
	\end{align}
	
	$\bullet$ Similarly, following the outlined procedure and utilizing the Burkholder-Davis-Gundy inequality, we derive the following estimate for $\mathcal{J}_{5}$.
	
	\begin{align}\label{j5}
		\mathcal{J}_{5}=& \frac{ 5^{p-1}}{E_{p\lambda-p+1}(\gamma t^{p\lambda-p+1})}E\bigg\Vert\int_{0}^{t}\mathcal{E}^{h_{1},h_{2}}_{\lambda,\lambda}(A_{0},A_{1},A_{2};t-s)\sigma(s,y(s),y(s-h_{1}),y(s-h_{2}))dW(s)\bigg\Vert^{p}\nonumber\\
		\leq&\frac{ 5^{p-1}C_p}{E_{p\lambda-p+1}(\gamma t^{p\lambda-p+1})}E\bigg(\int_{0}^{t}\Vert\mathcal{E}^{h_{1},h_{2}}_{\lambda,\lambda}(A_{0},A_{1},A_{2};t-s)\sigma(s,y(s),y(s-h_{1}),y(s-h_{2}))\Vert^{2} ds\bigg)^{\frac{p}{2}}\nonumber\\
		\leq&\frac{ 5^{p-1}C_p T^{\frac{p-2}{2}}}{E_{p\lambda-p+1}(\gamma t^{p\lambda-p+1})}E\int_{0}^{t}\Vert\mathcal{E}^{h_{1},h_{2}}_{\lambda,\lambda}(A_{0},A_{1},A_{2};t-s)\Vert^{p} \Vert\sigma(s,y(s),y(s-h_{1}),y(s-h_{2}))\Vert^{p} ds\nonumber\\
		\leq&\frac{ 10^{p-1}T^{\frac{p-2}{2}}}{E_{p\lambda-p+1}(\gamma t^{p\lambda-p+1})}\int_{0}^{t}\Vert\mathcal{E}^{h_{1},h_{2}}_{\lambda,\lambda}(A_{0},A_{1},A_{2};t-s)\Vert^{p} E\Vert \sigma(s,y(s),y(s-h_{1}),y(s-h_{2})) \nonumber\\
		-&\sigma(s,0,0,0)\Vert^{p}ds+\frac{ 10^{p-1}T^{\frac{p-2}{2}}}{E_{p\lambda-p+1}(\gamma t^{p\lambda-p+1})}\int_{0}^{t}\Vert\mathcal{E}^{h_{1},h_{2}}_{\lambda,\lambda}(A_{0},A_{1},A_{2};t-s)\Vert^{p} E\Vert \sigma(s,0,0,0)\Vert^{p}ds \nonumber\\
		\leq&\frac{ 10^{p-1}T^{\frac{p-2}{2}}}{E_{p\lambda-p+1}(\gamma t^{p\lambda-p+1})}L_{\sigma}\int_{0}^{t}\Vert\mathcal{E}^{h_{1},h_{2}}_{\lambda,\lambda}(A_{0},A_{1},A_{2};t-s)\Vert^{p}\frac{ E_{p\lambda-p+1}(\gamma s^{p\lambda-p+1})}{E_{p\lambda-p+1}(\gamma s^{p\lambda-p+1})} E\Vert y(s)\Vert^{p}ds\nonumber\\		
		+&\frac{ 10^{p-1}T^{\frac{p-2}{2}}}{E_{p\lambda-p+1}(\gamma t^{p\lambda-p+1})}L_{\sigma}\int_{0}^{t}\Vert\mathcal{E}^{h_{1},h_{2}}_{\lambda,\lambda}(A_{0},A_{1},A_{2};t-s)\Vert^{p} \frac{ E_{p\lambda-p+1}(\gamma s^{p\lambda-p+1})}{E_{p\lambda-p+1}(\gamma s^{p\lambda-p+1})} E\Vert y(s-h_{1})\Vert^{p}ds\nonumber\\
		+&\frac{ 10^{p-1}T^{\frac{p-2}{2}}}{E_{p\lambda-p+1}(\gamma t^{p\lambda-p+1})}L_{\sigma}\int_{0}^{t}\Vert\mathcal{E}^{h_{1},h_{2}}_{\lambda,\lambda}(A_{0},A_{1},A_{2};t-s)\Vert^{p} \frac{ E_{p\lambda-p+1}(\gamma s^{p\lambda-p+1})}{E_{p\lambda-p+1}(\gamma s^{p\lambda-p+1})} E\Vert y(s-h_{2})\Vert^{p}ds\nonumber\\
		+&\frac{ 10^{p-1}T^{\frac{p-2}{2}}}{E_{p\lambda-p+1}(\gamma t^{p\lambda-p+1})}\int_{0}^{t}\Vert\mathcal{E}^{h_{1},h_{2}}_{\lambda,\lambda}(A_{0},A_{1},A_{2};t-s)\Vert^{p} E\Vert \sigma(s,0,0,0)\Vert^{p}ds \nonumber\\
		\leq& \frac{3{\cdot}10^{p-1}T^{\frac{p-2}{2}}}{E_{p\lambda-p+1}(\gamma t^{p\lambda-p+1})}L_{\sigma}\int_{0}^{t}\Vert\mathcal{E}^{h_{1},h_{2}}_{\lambda,\lambda}(A_{0},A_{1},A_{2};t-s)\Vert^{p}E_{p\lambda-p+1}(\gamma t^{p\lambda-p+1})ds\nonumber\\
		\times&\max_{0\leq t\leq T}\frac{E\Vert \bar{y}^*\Vert^{p}}{E_{p\lambda-p+1}(\gamma t^{p\lambda-p+1})}+\frac{ 10^{p-1}T^{\frac{p-2}{2}}}{E_{p\lambda-p+1}(\gamma t^{p\lambda-p+1})}\mathcal{ M}_{4}\mathcal{F}\int_{0}^{t}(t-s)^{p\lambda-p}ds \nonumber\\
		\leq&\frac{3{\cdot}10^{p-1}T^{\frac{p-2}{2}}}{E_{p\lambda-p+1}(\gamma t^{p\lambda-p+1})}\mathcal{ M}_{4}L_{\sigma}\Vert \bar{y}\Vert_{\gamma}^{p}\int_{0}^{t}(t-s)^{p\lambda-p}E_{p\lambda-p+1}(\gamma t^{p\lambda-p+1})ds
		+\frac{ 10^{p-1}T^{p\lambda}}{p\lambda-p+1} \mathcal{ M}_{4}\mathcal{F} \nonumber\\
		\leq&\frac{3{\cdot}10^{p-1}T^{\frac{p-2}{2}}\mathcal{ M}_{4}L_{\sigma}\Gamma(p\lambda-p+1)}{\gamma}\Vert \bar{y}\Vert_{\gamma}^{p}+\frac{ 10^{p-1}T^{p\lambda}}{p\lambda-p+1} \mathcal{ M}_{4}\mathcal{F} 
	\end{align}
	
	By employing the approximations provided in equations (\ref{j0}-\ref{j5}), we obtain
	\begin{align}\label{s1}
		&\frac{E\Vert (\mathscr{L}y)(t)\Vert^{p}}{E_{p\lambda-p+1}(\gamma t^{p\lambda-p+1})}\leq5^{p-1}\mathcal{ M}_{0}\Vert\varphi(0)-A_{2}\varphi(-h_{2})\Vert^{p}_{\gamma}+ 5^{p-1}\mathcal{ M}_{2} \Phi^{p} \Vert A_{1}\Vert^{p} h^{p-1}_{1}\nonumber\\
		&+5^{p-1}\mathcal{ M}_{3} \Phi^{p} \Vert A_{2}\Vert^{p} h^{p-1}_{2}+\frac{3{\cdot}10^{p-1}T^{p-1}\mathcal{ M}_{4}L_{f}\Gamma(p\lambda-p+1)}{\gamma}\Vert \bar{y}\Vert_{\gamma}^{p}
		+\frac{ 10^{p-1}T^{p\lambda}}{p\lambda-p+1} \mathcal{ M}_{4}\mathcal{F}\nonumber\\ &+\frac{3{\cdot}10^{p-1}T^{\frac{p-2}{2}}\mathcal{ M}_{4}L_{\sigma}\Gamma(p\lambda-p+1)}{\gamma}\Vert \bar{y}\Vert_{\gamma}^{p}+\frac{ 10^{p-1}T^{p\lambda}}{p\lambda-p+1} \mathcal{ M}_{4}\mathcal{F} 
	\end{align}
	By analyzing the preceding information, one can infer the existence of a positive constant \(k\) such that
	\[
	\|\Vert{\mathscr{L}}(t)\Vert^{p} \leq k(1 + \Vert\bar{y}\Vert^{p}(t)).
	\]
	This implies that the mapping of \(\mathscr{L}\) preserves the set \(H^{p}_{\varphi}\). 
	\bigskip
	
	$\bullet$ \textbf{Step 2.}We demonstrate that $\mathscr{L}$ functions as a contraction mapping. To establish this, it is essential to prove that $\mathscr{L}$ is contractive on the set $H^{p}_{\varphi}$. Consider any pair of elements $y, z \in H^{p}_{\varphi}$. Observe that for all $t \in [0, T]$, the expression:
	
	\begin{align}
		& (\mathscr{L}y)(t) - (\mathscr{L}z)(t) \nonumber\\
		=& \int_{0}^{t}\mathcal{E}^{h_{1},h_{2}}_{\lambda,\lambda}(A_{0},A_{1},A_{2};t-s)\Big(f(s,y(s),y(s-h_{1}),y(s-h_{2})) - f(s,z(s),z(s-h_{1}),z(s-h_{2}))\Big)ds \nonumber\\
		+& \int_{0}^{t}\mathcal{E}^{h_{1},h_{2}}_{\lambda,\lambda}(A_{0},A_{1},A_{2};t-s)\Big(\sigma(s,y(s),y(s-h_{1}),y(s-h_{2})) - \sigma(s,z(s),z(s-h_{1}),z(s-h_{2}))\Big)dW(s).
	\end{align}
	
	By using the Jensen's inequality, we deduced that:
	\begin{align}
		&\frac{E\Vert  (\mathscr{L}y)(t)-(\mathscr{L}z)(t)\Vert^{p}}{E_{p\lambda-p+1}(\gamma t^{p\lambda-p+1})}\nonumber\\
		&\leq      \frac{2^{p-1}}{E_{p\lambda-p+1}(\gamma t^{p\lambda-p+1})}E\bigg\Vert\int_{0}^{t}\mathcal{E}^{h_{1},h_{2}}_{\lambda,\lambda}(A_{0},A_{1},A_{2};t-s)\big(f(s,y(s),y(s-h_{1}),y(s-h_{2}))\nonumber\\
		&-f(s,z(s),z(s-h_{1}),z(s-h_{2}))\big)ds\bigg\Vert^{p}\nonumber\\
		&+\frac{2^{p-1}}{E_{p\lambda-p+1}(\gamma t^{p\lambda-p+1})}E\bigg\Vert\int_{0}^{t}\mathcal{E}^{h_{1},h_{2}}_{\lambda,\lambda}(A_{0},A_{1},A_{2};t-s)\big(\sigma(s,y(s),y(s-h_{1}),y(s-h_{2}))\nonumber\\
		&-\sigma(s,z(s),z(s-h_{1}),z(s-h_{2}))\big)dW(s)\bigg\Vert^{p}=\mathcal{J}_{6}+J_{7}.
	\end{align}
	$\bullet$ By employing Hölder's inequality, $(\mathcal{A}_{1})$, Jensen's inequality, Lemma \ref{l2}, $(\mathcal{A}_{2})$, and Lemma \ref{l1}, one can attain the estimation of $\mathcal{J}_{6}$.
	\begin{align}\label{19}
		\mathcal{J}_{6}=&\frac{2^{p-1}}{E_{p\lambda-p+1}(\gamma t^{p\lambda-p+1})}E\bigg\Vert\int_{0}^{t}\mathcal{E}^{h_{1},h_{2}}_{\lambda,\lambda}(A_{0},A_{1},A_{2};t-s)\big(f(s,y(s),y(s-h_{1}),y(s-h_{2}))\nonumber\\
		-&f(s,z(s),z(s-h_{1}),z(s-h_{2}))\big)ds\bigg\Vert^{p}\nonumber\\
		\leq&\frac{6^{p-1}L^{p}_{f}T^{p-1}}{E_{p\lambda-p+1}(\gamma t^{p\lambda-p+1})}\int_{0}^{t}\Vert \mathcal{E}^{h_{1},h_{2}}_{\lambda,\lambda}(A_{0},A_{1},A_{2};t-s)\Vert^{p} E\Vert y(s)-z(s)\Vert^{p}ds\nonumber\\
		+&\frac{6^{p-1}L^{p}_{f}T^{p-1}}{E_{p\lambda-p+1}(\gamma t^{p\lambda-p+1})}\int_{0}^{t}\Vert \mathcal{E}^{h_{1},h_{2}}_{\lambda,\lambda}(A_{0},A_{1},A_{2};t-s)\Vert^{p} E\Vert y(s-h_1)-z(s-h_1)\Vert^{2}ds\nonumber\\
		+&\frac{6^{p-1}L^{p}_{f}T^{p-1}}{E_{p\lambda-p+1}(\gamma t^{p\lambda-p+1})}\int_{0}^{t}\Vert \mathcal{E}^{h_{1},h_{2}}_{\lambda,\lambda}(A_{0},A_{1},A_{2};t-s)\Vert^{p} E\Vert y(s-h_2)-z(s-h_2)\Vert^{p}ds\nonumber\\
		\leq&\frac{6^{p-1}L^{p}_{f}T^{p-1}\mathcal{M}_{4}}{E_{p\lambda-p+1}(\gamma t^{p\lambda-p+1})} \int_{0}^{t}(t-s)^{p\lambda-p} E_{p\lambda-p+1}(\gamma s^{p\lambda-p+1})ds\max_{0\leq t\leq T}\frac{E\Vert \bar{y}^{*}(t)-\bar{z}^{*}(t)}{E_{p\lambda-p+1}(\gamma s^{p\lambda-p+1})}\nonumber\\
		\leq&\frac{6^{p-1}L^{p}_{f}T^{p-1}\mathcal{M}_{4}\Gamma(p\lambda-p+1)}{\gamma}\Vert \bar{y}-\bar{z}\Vert^{p}_{\gamma}
	\end{align}
	$\bullet$ Similarly, we will obtain for $\mathcal{J}_{7}$, in the following form:
	\begin{align}\label{20}
		&\mathcal{J}_{7}\leq\frac{6^{p-1}L^{\sigma}_{\sigma}T^{\frac{p-2}{2}}}{E_{p\lambda-p+1}(\gamma t^{p\lambda-p+1})}\int_{0}^{t}\Vert \mathcal{E}^{h_{1},h_{2}}_{\lambda,\lambda}(A_{0},A_{1},A_{2};t-s)\Vert^{p} E\Vert y(s)-z(s)\Vert^{p}ds\nonumber\\
		+&\frac{6^{p-1}L^{p}_{\sigma}T^{\frac{p-2}{2}}}{E_{p\lambda-p+1}(\gamma t^{p\lambda-p+1})}\int_{0}^{t}\Vert \mathcal{E}^{h_{1},h_{2}}_{\lambda,\lambda}(A_{0},A_{1},A_{2};t-s)\Vert^{p} E\Vert y(s-h_1)-z(s-h_1)\Vert^{2}ds\nonumber\\
		+&\frac{6^{p-1}L^{p}_{\sigma}T^{\frac{p-2}{2}}}{E_{p\lambda-p+1}(\gamma t^{p\lambda-p+1})}\int_{0}^{t}\Vert \mathcal{E}^{h_{1},h_{2}}_{\lambda,\lambda}(A_{0},A_{1},A_{2};t-s)\Vert^{p} E\Vert y(s-h_2)-z(s-h_2)\Vert^{p}ds\nonumber\\
		\leq&\frac{6^{p-1}L^{p}_{\sigma}T^{\frac{p-2}{2}}\mathcal{M}_{4}}{E_{p\lambda-p+1}(\gamma t^{p\lambda-p+1})} \int_{0}^{t}(t-s)^{p\lambda-p} E_{p\lambda-p+1}(\gamma s^{p\lambda-p+1})ds\max_{0\leq t\leq T}\frac{E\Vert \bar{y}^{*}(t)-\bar{z}^{*}(t)}{E_{p\lambda-p+1}(\gamma s^{p\lambda-p+1})}\nonumber\\
		\leq&\frac{6^{p-1}L^{p}_{\sigma}T^{\frac{p-2}{2}}\mathcal{M}_{4}\Gamma(p\lambda-p+1)}{\gamma}\Vert \bar{y}-\bar{z}\Vert^{p}_{\gamma}.
	\end{align}
	
	By using (\ref{19}) and(\ref{20}), we will achieve the following outcomes.
	\begin{align}\label{13}
		\frac{E\Vert  (\mathscr{L}y)(t)-(\mathscr{L}z)(t)\Vert^{p}}{E_{p\lambda-p+1}(\gamma t^{p\lambda-p+1})}\leq&\bigg(\frac{6^{p-1}(L^{p}_{f}T^{p-1}+L^{p}_{\sigma}T^{\frac{p-2}{2}})\mathcal{ M}_{4}\Gamma(p\lambda-p+1)}{\gamma}\bigg)\Vert \bar{y}-\bar{z}\Vert^{p}_{\gamma}\nonumber\\
		=&K\Vert \bar{y}-\bar{z}\Vert^{p}_{\gamma}.
	\end{align}
	By taking weighted maximum norm (\ref{m1}), we achieve
	\begin{align}
		\Vert  (\mathscr{L}y)(t)-(\mathscr{L}z)(t)\Vert^{p}_{\gamma}\leq K\Vert \bar{y}-\bar{z}\Vert^{p}_{\gamma}
	\end{align}
	Since $K < 1$ in accordance with $(\mathcal{A}_{3})$, the operator $\mathscr{L}$ is acknowledged as a contraction mapping on $H^{p}_{\varphi}$. Consequently, there exists a unique fixed point $y \in H^{p}_{\varphi}$, establishing it as the solution to (\ref{1}).
\end{proof}
\subsection{Finite-time stability of the system (\ref{1})}
In this section of the article, we intend to demonstrate the finite-time stability of equation (\ref{1}) through the following theorem. Before presenting the theorem, we define finite-time stability of system as follows:
\begin{definition}\label{d1}\cite{42} Given positive numbers $\Lambda, \varepsilon$ satisfying $\Lambda < \varepsilon$, the system is finite-time stable if $\|\Phi\|_{\gamma} \leq \Lambda$ implies $\|y\|_{\gamma} \leq \varepsilon$ with respect to $\{\Lambda, \varepsilon, -h, T\}$, for all $t \in [-h, T]$.
\end{definition}
Now, we prove the theorem concerning the finite-time stability of the system described by equation (\ref{1}) over the given interval.
\begin{theorem}
	Assuming that Assumptions ($\mathcal{A}_{1}-\mathcal{A}_{4}$) hold, and there exist two positive numbers $\Lambda$, $\varepsilon$ satisfying $\Lambda < \varepsilon$, and $\|\Phi\|_{\gamma} < \Lambda$, then system (\ref{1}) is finite-time stable on $[-h,T]$ provided that
	\begin{align*}
		\Lambda = & \frac{\varepsilon}{\tilde{M} e^{C(3T-h_{1}-h_{2})}} - 5^{p-1}(L_{f}+C_p L_{\sigma})T^{p}\mathcal{ M}_{4},
	\end{align*}
	where $\tilde{M} = 5^{p-1}\mathcal{ M}_{1}(1+\Vert A_{2}\Vert^{p})+5^{p-1}\Big(\mathcal{ M}_{2}\Vert A_{1}\Vert^{p}+\mathcal{ M}_{3}\Vert A_{2}\Vert^{p}\Big)h^{p-1} + 2C h^{\frac{p-1}{p}}.$
\end{theorem}
\begin{proof}
	From Equation (\ref{1}), for \( t \in [0, T]\), utilizing Jensen's inequality, Hölder inequality, Assumption ($\mathcal{A}_{4}$), and Burkholder-Davis-Gundy inequality, we obtain:
	\begin{align*}
		&E(\Vert y(t)\Vert^{p})\leq 5^{p-1}\mathcal{ M}_{1}E\Vert \varphi(0)-A_2 \varphi(-h_2)\Vert^{p}+5^{p-1}\mathcal{ M}_{2}\Phi^{p}\Vert A_{1}\Vert^{p}h^{p-1}_{1}+5^{p-1}\mathcal{ M}_{3}\Phi^{p}\Vert A_{2}\Vert^{p}h^{p-1}_{2}\\
		&+5^{p-1}t^{p-1}\mathcal{ M}_{4}L_{f}E\int_{0}^{t}(1+\Vert y(s)\Vert^{p}+\Vert y(s-h_{1})\Vert^{p}+\Vert y(s-h_{2})\Vert^{p})ds\\
		&+5^{p-1}C_{p}t^{\frac{p-2}{2}}\mathcal{ M}_{4}L_{\sigma}E\int_{0}^{t}(1+\Vert y(s)\Vert^{p}+\Vert y(s-h_{1})\Vert^{p}+\Vert y(s-h_{2})\Vert^{p})ds\\
		&\leq 5^{p-1}\mathcal{ M}_{1}E\Vert \varphi(0)-A_2 \varphi(-h_2)\Vert^{p}+5^{p-1}\mathcal{ M}_{2}\Phi^{p}\Vert A_{1}\Vert^{p}h^{p-1}+5^{p-1}\mathcal{ M}_{3}\Phi^{p}\Vert A_{2}\Vert^{p}h^{p-1}\\
		&+5^{p-1}t^{p}\mathcal{ M}_{4}L_{f}+5^{p-1}T^{p-1}\mathcal{ M}_{4}L_{f}\int_{0}^{t}E(\Vert y(s)\Vert^{p}+\Vert y(s-h_{1})\Vert^{p}+\Vert y(s-h_{2})\Vert^{p})ds\\
		&+5^{p-1}C_{p}t^{\frac{p}{2}}\mathcal{ M}_{4}L_{\sigma}+5^{p-1}C_{p}T^{\frac{p-2}{2}}\mathcal{ M}_{4}L_{\sigma}\int_{0}^{t}E(\Vert y(s)\Vert^{p}+\Vert y(s-h_{1})\Vert^{p}+\Vert y(s-h_{2})\Vert^{p})ds\\
		&\leq 5^{p-1}\mathcal{ M}_{1}E\Vert \varphi(0)-A_2 \varphi(-h_2)\Vert^{p}+5^{p-1}\Big(\mathcal{ M}_{2}\Vert A_{1}\Vert^{p}+\mathcal{ M}_{3}\Vert A_{2}\Vert^{p}\Big)\Phi^{p}h^{p-1}
		+5^{p-1}(L_{f}+C_p L_{\sigma})t^{p}\mathcal{ M}_{4}\\
		&+5^{p-1}(L_{f}+C_p L_{\sigma})t^{p-1}\mathcal{ M}_{4}\int_{0}^{t}\Big(E\big(\Vert y(s)\Vert^{p}\big)+E\big(\Vert y(s-h_{1})\Vert^{p}\big)+E\big(\Vert y(s-h_{2})\Vert^{p}\big)\Big)ds\\
	\end{align*}
	Letting $\beta(t)=5^{p-1}\mathcal{ M}_{1}E\Vert \varphi(0)-A_2 \varphi(-h_2)\Vert^{p}+5^{p-1}\Big(\mathcal{ M}_{2}\Vert A_{1}\Vert^{p}+\mathcal{ M}_{3}\Vert A_{2}\Vert^{p}\Big)\Phi^{p}h^{p-1}+5^{p-1}(L_{f}+C_p L_{\sigma})t^{p}\mathcal{ M}_{4}$ and $C=5^{p-1}(L_{f}+C_p L_{\sigma})t^{p-1}\mathcal{ M}_{4}$, then we will get
	\begin{align}\label{qt}
		E(\Vert y(t)\Vert^{p})\leq&\beta(t)+C\int_{0}^{t}\Big(E\big(\Vert y(s)\Vert^{p}\big)+E\big(\Vert y(s-h_{1})\Vert^{p}\big)+E\big(\Vert y(s-h_{2})\Vert^{p}\big)\Big)ds
	\end{align}
	By applying Lemma \ref{bay} and utilizing the inequality (\ref{qt}) with $y(t)=\varphi(t)$ for $t\in [-h,0]$ as the initial condition, we obtain
	\begin{align*}
		E(\Vert y(t)\Vert^{p})\leq&\bigg[\beta(t)+C\sum_{i=1}^{2}\int_{[0,t]\cap [0,h_i]}E\Vert\varphi(s-h_{i})\Vert^{p}ds\bigg] e^{Ct+\sum_{i=1}^{2}\int_{[0,t]\setminus [0,h_i]} C ds}\\
		\leq&\bigg[\beta(t)+2C\int_{-h}^{0}E\Vert\varphi(s)\Vert^{p}ds\bigg] e^{C(3t-h_1-h_2)}
		\leq\bigg[\beta(t)+2CE(\Phi^{p}) h^{\frac{p-1}{p}}\bigg] e^{C(3T-h_1-h_2)}\\
		=&\bigg[5^{p-1}\mathcal{ M}_{1}E\Vert \varphi(0)-A_2 \varphi(-h_2)\Vert^{p}+5^{p-1}\Big(\mathcal{ M}_{2}\Vert A_{1}\Vert^{p}+\mathcal{ M}_{3}\Vert A_{2}\Vert^{p}\Big)E(\Phi^{p})h^{p-1}\\
		+&5^{p-1}(L_{f}+C_p L_{\sigma})t^{p}\mathcal{ M}_{4}+2CE(\Phi^{p}) h^{\frac{p-1}{p}}\bigg] e^{C(3T-h_1-h_2)}\\
		\leq&\bigg[5^{p-1}\mathcal{ M}_{1}(1+\Vert A_{2}\Vert^{p})E(\Phi^{p})+5^{p-1}\Big(\mathcal{ M}_{2}\Vert A_{1}\Vert^{p}+\mathcal{ M}_{3}\Vert A_{2}\Vert^{p}\Big)E(\Phi^{p})h^{p-1}\\
		+&5^{p-1}(L_{f}+C_p L_{\sigma})T^{p}\mathcal{ M}_{4}+2CE(\Phi^{p}) h^{\frac{p-1}{p}}\bigg] e^{C(3T-h_1-h_2)}\\
		=&\bigg(\tilde{M}E(\Phi^{p})+5^{p-1}(L_{f}+C_p L_{\sigma})T^{p}\mathcal{ M}_{4}\bigg)e^{C(3T-h_1-h_2)}
	\end{align*}
	where $\tilde{M}=5^{p-1}\mathcal{ M}_{1}(1+\Vert A_{2}\Vert^{p})+5^{p-1}\Big(\mathcal{ M}_{2}\Vert A_{1}\Vert^{p}+\mathcal{ M}_{3}\Vert A_{2}\Vert^{p}\Big)h^{p-1}
	+2C h^{\frac{p-1}{p}}.$
	
	Therefore, due to the norm of $\Vert \cdot\Vert_{\gamma}$, we get
	\begin{align*}
		\frac{	E(\Vert y(t)\Vert^{p})}{E_{p\lambda-p+1}(\gamma t^{p\lambda-p+1})}\leq&\frac{1}{E_{p\lambda-p+1}(\gamma t^{p\lambda-p+1})}\bigg(\tilde{M}E(\Phi^{p})+5^{p-1}(L_{f}+C_p L_{\sigma})T^{p}\mathcal{ M}_{4}\bigg)e^{C(3T-h_1-h_2)}
	\end{align*}
	Therefore, according to the weight maximum norm and Eq. (\ref{1}), we
	have
	\begin{align*}
		\Vert y\Vert_{\gamma}\leq&\big(\tilde{M} \Vert \Phi\Vert_{\gamma}+5^{p-1}(L_{f}+C_p L_{\sigma})T^{p}\mathcal{ M}_{4}\big)e^{C(3T-h_1-h_2)}\\
		\leq&\big(\tilde{M} \Lambda+5^{p-1}(L_{f}+C_p L_{\sigma})T^{p}\mathcal{ M}_{4}\big)e^{C(3T-h_1-h_2)}\leq\varepsilon\\
	\end{align*}
	Using definition \ref{d1}, we can deduce that the system (\ref{1}) is finite-time stable over the interval $[-h,T]$.  
\end{proof}

\section{Examples}

In this section, we present various instances to illustrate the outcomes established in the preceding section.
\begin{example}
	Consider the fractional stochastic neutral delay differential equation given by:
	
	\[
	\begin{cases}
		{^{C}}D^{0.5}_0 y(t)= A_{0}y(t) + A_{1}y(t - 1)+ A_{2}D^{0.5}_0 y(t - 0.5) + \cos(y(t-1))+\sin(y(t-0.5))\frac{dW(t)}{dt}, \quad t\in[0,2],\\
		y(t) = 0, \quad t \in [-1, 0].
	\end{cases}
	\]\label{e1}
	
	Where the matrices are defined as:
	
	\[
	\begin{aligned}
		A_{0}=\begin{pmatrix}
			-1&2\\
			0&1\\
		\end{pmatrix}, \quad A_{1}=\begin{pmatrix}
			2&4\\
			1&0\\
		\end{pmatrix}, \quad A_{2}=\begin{pmatrix}
			3&0.5\\
			0&-2\\
		\end{pmatrix}.
	\end{aligned}
	\]
	
	These matrices satisfy the conditions:
	
	\[
	\begin{aligned}
		A_0 A_1 &\neq A_1 A_0, \\
		A_0 A_2 &\neq A_2 A_0, \\
		A_1 A_2 &\neq A_2 A_1.
	\end{aligned}
	\]
	
	Utilizing the definition of \ref{d1}, we obtain the following solution for the above equation:
	
	\[
	\begin{aligned}
		y(t)=&\int_{0}^{t}\mathcal{E}_{0.5, 0.5}^{1, 0.5}(A_{0}, A_{1}, A_{2}, t-s)\cos(y(s-1))ds\\
		&+\int_{0}^{t}\mathcal{E}_{0.5, 0.5}^{1, 0.5}(A_{0}, A_{1}, A_{2}, t-s)\sin(y(s-0.5))dW(s).
	\end{aligned}
	\]
	
	This results in:
	
	\[
	\mathcal{E}_{0.5,0.5}^{1 ,0.5}(A_0, A_1, A_2;t) :=
	\begin{cases}
		\emptyset, \quad -h \leq t < 0, \\
		I, \quad t = 0, \\
		\sum_{k=0}^{\infty} \sum_{m_1=0}^{\infty} \sum_{m_2=0}^{\infty} Q_{k+1}( m_1 , 0.5 m_2 )\frac{(t - m_1 -0.5 m_2 )_{+}^{0.5 k-0.5}}{\Gamma(0.5 k+0.5)} , & t \in \mathbb{R}^+,
	\end{cases}
	\]
	
	where:
	
	\[
	(t - m_1  -0.5 m_2 )_+ =
	\begin{cases}
		t - m_1  - 0.5 m_2, & t \geq  m_1 + 0.5 m_2 , \\
		0, & t < m_1  + 0.5 m_2.
	\end{cases}
	\]	
	
	Furthermore, the sequence $Q_{k+1}(m_{1}, 0.5 m_{2})$ is governed by the recursive formula:
	
	\[
	\begin{aligned}
		Q_{k+1}(m_{1}, 0.5 m_{2})= &A_{0}Q_{k} (m_{1}, 0.5 m_{2})+A_{1} Q_{k} (m_{1} -1, 0.5 m_{2})\\
		&+ A_{2} Q_{k+1}(m_{1}, 0.5(m_{2} -1)), \quad k = 1, 2, \ldots,
	\end{aligned}
	\]
	
	with initial conditions:
	
	\[
	\begin{aligned}
		Q_{0}(m_{1}, 0.5 m_{2} ) &= Q_{k} (-1, 0.5 m_{2} ) = Q_{k} (m_{1}, -0.5) =\emptyset, \quad k = 0, 1, 2, \ldots,\\
		Q_{1}(0, 0) &= I, \quad Q_{1}(m_{1}, 0.5 m_{2}) =\emptyset.
	\end{aligned}
	\]
	
	Here, $I \in R^{2\times 2}$ and $\emptyset\in R^{2\times 2}$ represent identity and zero matrices, respectively.   
	
	Certainly, the Lipschitz condition for the given equation is easily verified by examining the following expressions:
	
	1. The Lipschitz condition for the function \(f(t, y(t), y(t-0.5), y(t-1))\) can be checked in the form:
	
	\begin{align*}
		&\Vert f(t,y(t),y(t-0.5),y(t-1))-f(t,z(t),z(t-0.5),z(t-1))\Vert^{p}\\
		&=\Vert \cos(y(t-1))-\cos(z(t-1))\Vert^{p}\\
		&=\bigg\Vert -2 \sin \big(\frac{y(t-1)-z(t-1)}{2}\big) \sin  \big(\frac{y(t-1)+z(t-1)}{2}\big)\bigg\Vert^{p}\\
		&\leq \Vert y(t-1)-z(t-1)\Vert^{p},
	\end{align*}
	
	This implies \(L_f = 1\).
	
	2. Similarly, the Lipschitz condition for the function \(\sigma(t, y(t), y(t-0.5), y(t-1))\) is verified as follows:
	
	\begin{align*}
		&\Vert \sigma(t,y(t),y(t-0.5),y(t-1))-\sigma(t,z(t),z(t-0.5),z(t-1))\Vert^{p}\\
		&=\Vert \sin(y(t-0.5))-\sin(z(t-0.5))\Vert^{p}\\
		&=\bigg\Vert 2 \sin \big(\frac{y(t-0.5)-z(t-0.5)}{2}\big) \cos  \big(\frac{y(t-0.5)+z(t-0.5)}{2}\big)\bigg\Vert^{p}\\
		&\leq \Vert y(t-0.5)-z(t-0.5)\Vert^{p}.
	\end{align*}
	
	This implies \(L_\sigma = 1\).
	
	Therefore, it can be concluded that both \(L_f\) and \(L_\sigma\) are equal to 1 based on the derived inequalities.
	This indicates that under the conditions ($\mathcal{A}_{1}-\mathcal{A}_{4}$), with Lipschitz constant \(L_f =L_{\sigma}=1\), Theorems  3.1 and 3.2 guarantee the existence and uniqueness of a continuous solution for Equation (\ref{1}). Moreover, this solution is stable in the Finite time sense on the interval \([-1, 2]\).
	
	In Figure 1, the simulation displays the solution \(y(\cdot) \in C^{1}([-1, 2], \mathbb{R}^{2}_{+})\) of the Cauchy problem (\ref{1}), considering a fractional-order neutral delay differential equation with an initial data \(\varphi(\cdot)=0 \).
	\begin{figure}[h]
		\centering
		\includegraphics[width=0.6\textwidth]{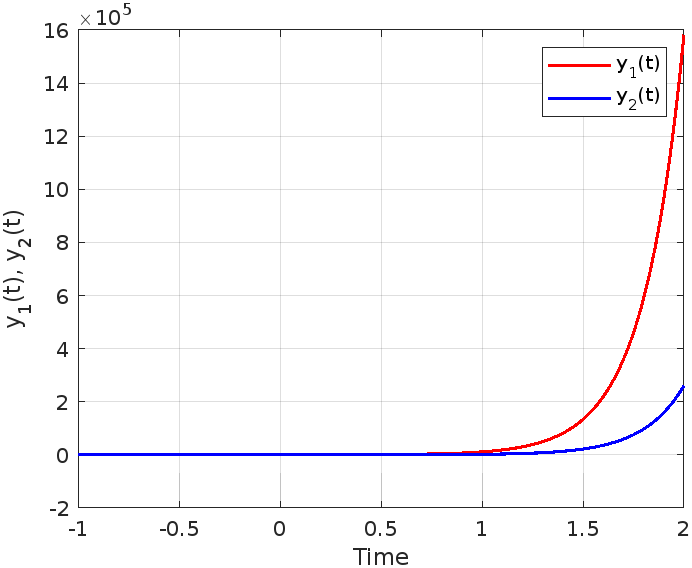}
		\caption{\\
		}\label{}
	\end{figure}
\end{example}
\begin{example}
	Consider the following stochastic fractional delay differential equation.
	\begin{equation}\label{example_equation}
		\begin{cases}
			{^{C}}D^{0.5}_0 y(t)= A_{0}y(t) + A_{1}y(t - 0.5)+ A_{2}D^{0.5}_0 y(t - 1) + \sin(t + y(t - 0.5) + y(t - 1))\\
			+\cos(t + y(t - 0.5) + y(t - 1))\frac{dW(t)}{dt}, \quad t\in[0,10],\\
			y(t) = \varphi(t), \quad t \in [-1, 0].
		\end{cases}
	\end{equation}
	
	In this example:
	
	$\bullet$ \( {^{C}}D^{0.5}_0 y(t) \) is the Caputo fractional derivative of the function \( y(t) \).
	
	$\bullet$ The linear terms \( A_{0}y(t) \), \( A_{1}y(t - 0.5) \), and \( A_{2}D^\lambda_0 y(t - 1) \) are included.
	
	$\bullet$ The nonlinear term \( \sin(t + y(t - 0.5) + y(t - 1)) \) represents a nonlinear function of the system.
	
	$\bullet$ The stochastic term \(\cos(t + y(t - 0.5) + y(t - 1))\frac{dW(t)}{dt} \) involves a Wiener process \( W(t) \).
	
	$\bullet$ The initial condition is set to \begin{align*}
		y(t)=\varphi(t)=e^{t}.
	\end{align*}
	
	Where the matrices are defined as:
	
	\[
	\begin{aligned}
		A_{0}=\begin{pmatrix}
			-1&2\\
			0&1\\
		\end{pmatrix}, \quad A_{1}=\begin{pmatrix}
			2&4\\
			1&0\\
		\end{pmatrix}, \quad A_{2}=\begin{pmatrix}
			3&0.5\\
			0&-2\\
		\end{pmatrix}.
	\end{aligned}
	\]
	Given the parallel conditions observed in the first example, where matrices yielded the same result, and noting that the Mittag-Leffler function adheres to analogous conditions as illustrated in Example 4.1, we can express the solution for the equation using the definition denoted by Equation (1).
	\begin{align*}
		y(t)=&\mathcal{E}^{0.5,1}_{0.5,1}(A_{0},A_{1},A_{2};t)(1-A_{2}e^{-1})+\int_{-0.5}^{0}\mathcal{E}^{0.5,1}_{0.5,0.5}(A_{0},A_{1},A_{2};t-0.5-s)A_{1}e^{s}ds\\
		+&\int_{-1}^{0}\mathcal{E}^{0.5,1}_{0.5,0}(A_{0},A_{1},A_{2};t-1-s)A_{2}e^{s}ds\\
		+&\int_{0}^{t}\mathcal{E}^{0.5,1}_{0.5,0.5}(A_{0},A_{1},A_{2};t-s) \sin(s + y(s - 0.5) + y(s - 1))ds\\
		+&\int_{0}^{t}\mathcal{E}^{0.5,1}_{0.5,0.5}(A_{0},A_{1},A_{2};t-s) \cos(s + y(s - 0.5) + y(s - 1))dW(s)\\
		y(t)=& e^{t} \quad \text{if } t \in [-1, 0].
	\end{align*}
	If we confirm that the conditions $\mathcal{A}_{1}-\mathcal{A}_{4}$ hold for the specified equation, Theorems 3.1 and 3.2 guarantee the presence of a continuous solution for Equation (\ref{1}), ensuring both its existence and uniqueness. Furthermore, this solution exhibits finite-time stability within the interval $[-1,10]$.
\end{example}
\section{Conclusion}

This paper focuses on delivering comprehensive findings regarding the stability analysis of fractional stochastic neutral delay differential equations (FSNDDEs). We have established the finite-time stability for FSNDDEs by leveraging the existence and uniqueness of solutions. Our stability analyses rely on employing suitable weighted maximum norms and making assumptions on nonlinear terms consistent with finite-dimensional stochastic theory. Additionally, we have provided illustrative examples to reinforce the validity of our results.

\end{document}